\documentclass{article}

\usepackage{amsmath,yhmath,amsfonts,amssymb,amsthm,bbding}

\newcommand{\dist}{{\operatorname{dist}}}

\newcommand{\diam}{\operatorname{diam}}

\newcommand{\counte}{theorem}
\newtheorem{theorem}{\bf Theorem}[section]

\newtheorem{defn}[\counte]{\bf Definition}

\newtheorem{prop}[\counte]{\bf Proposition}
\newtheorem{lemma}[\counte]{\bf Lemma}

\newtheorem{coro}[\counte]{\bf Corollary}
\newtheorem{example}[\counte]{\bf Example}

\newtheorem{remark}[\counte]{\bf Remark}

\allowdisplaybreaks

\numberwithin{equation}{section}

\renewcommand{\thefootnote}{\fnsymbol{footnote}}

\begin{document}

\renewcommand{\thefootnote}{\arabic{footnote}}

\centerline{\bf\Large Quasi-convex subsets in Alexandrov spaces}
\centerline{\bf\Large with lower curvature bound \footnote{Supported by NSFC 11971057 and BNSF Z190003. \hfill{$\,$}}}

\vskip5mm

\centerline{Xiaole Su, Hongwei Sun, Yusheng Wang\footnote{The
corresponding author (E-mail: wyusheng@bnu.edu.cn). \hfill{$\,$}}}

\vskip6mm

\noindent{\bf Abstract.} In this paper, we introduce quasi-convex subsets
in Alxandrov spaces with lower curvature bound, which include
not only all closed convex subsets without boundary but also all extremal subsets.
Moreover, we explore several
essential properties of such kind of subsets including a generalized Liberman theorem.
It turns out that the quasi-convex subset is a nice and fundamental concept to illustrate the similarities and differences between
Riemannian manifolds and Alxandrov spaces with lower curvature bound.

\vskip1mm


\vskip1mm


\vskip6mm

\setcounter{section}{-1}

\section{Introduction}

\vskip2mm

As a generalization of Riemannian manifolds with lower sectional curvature bound, finite dimensional Alexandrov spaces with
lower curvature bound have been studied intensively in past 30 years by many geometrists, especially by
Professor Perel'man (\cite{BGP}, [Pe1-2], [PP1-2]; for the definition refer to \cite{BGP}).
In an Alexandrov space, a geodesic (locally shortest path) might not be prolonged even if
the space is complete and has no boundary. This is quite different from complete Riemannian manifolds.
As a result, the concept of `locally convex subset'
is the best one as a generalization of `totally geodesic submanifold' to show a kind of similarity between Alexandrov spaces and Riemannian manifolds (in a Riemannian manifold, a submanifold is locally convex if and only if it is totally geodesic).

Compared with locally convex subset, extremal subset is a concept to illustrate an
absolute difference between Alexandrov spaces and Riemannian manifolds,
because there is no  extremal subset in a complete Riemannian manifold at all.
The extremal subset was introduced in \cite{PP1} as follows.

\begin{defn}\label{def0.1}{\rm In an $X\in \text{Alex}(k)$, a closed subset $F$ is called to be {\it extremal }
if the following condition is satisfied: if the distance function $\dist_q$ from $q\not\in F$ restricted to $F$
attains a minimum at $p\in F$, then $p$ is a maximal-type critical point of $\dist_q$ on $X$, i.e.
\begin{equation} \label{eqn0.1}
\limsup_{p_i\in X, p_i \to p}\frac{|p_iq|-|pq|}{|p_ip|}\leqslant 0.
\end{equation} }
\end{defn}

In this paper, $\text{Alex}(k)$ denotes the set of all complete and finite dimensional Alexandrov spaces with
curvature $\geqslant k$, and  $X$ always denotes a space in  $\text{Alex}(k)$, and $|pq|$ denotes the distance between two points $p$ and $q$ in $X$.

In Definition \ref{def0.1}, $F$ might be the empty set, but
for some kind of completeness an extra condition is needed if $F$ is the empty set or a single point
and if $X$ is of curvature $\geqslant 1$ (see (0.2.3) below).

For example, if $X$ has non-empty boundary, which might not be locally convex in $X$, then the boundary $\partial X$ is an extremal subset in $X$ (for other examples refer to (\ref{eqn1.1}) and (\ref{eqn2.6}) below). It seems that there is no relation between locally convex subsets and extremal subsets. However, from the following equivalent version of Definition \ref{def0.1}, we can see some relation between them.

Denote by $[pq]$ a minimal geodesic (i.e. a shortest path) between $p$ and $q$, and denote by $\uparrow_p^q$ its direction from $p$ to $q$. For $[pq],[pr]\subset X$, we can define an angle at $p$ between them,
which is equal to the distance $|\uparrow_p^q\uparrow_p^r|$ in $\Sigma_pX$
(where $\Sigma_pX\in\text{Alex}(1)$ is the space of directions of $X$ at $p$, [BGP]).  From the definition of the angle,
it is almost obvious that (\ref{eqn0.1}) is equivalent to that
\begin{equation} \label{eqn0.2}
\forall\ r\in X\setminus\{p\},\ |\uparrow_p^q\uparrow_p^r|\leqslant \frac{\pi}{2} \text{ for any $[pq]$ and any $[pr]$.}
\end{equation}
We would like to point out that (\ref{eqn0.2}) was not formulated out in \cite{PP1} because of its obviousness.

On the other hand, if $F$ is a convex subset in $X$, then from a lemma by Yamaguchi (see (0.2.1) below)
one can conclude that if $\dist_q|_F$ with $q\not\in F$ attains a minimum at $p\in F^\circ$, where $F^\circ$ is
the interior part\footnote{A locally convex subset $F$ in $X\in \text{Alex}(k)$ also belongs to $\text{Alex}(k)$,
so we can talk about whether $F$ has a boundary or not.} of $F$, then
\begin{equation} \label{eqn0.3}
\forall\ r\in F\setminus\{p\},\ |\uparrow_p^q\uparrow_p^r|=\frac{\pi}{2} \text{ for any $[pq]$ and some $[pr]$}.
\end{equation}
(In fact, the $[pr]$ satisfying (\ref{eqn0.3}) can lie in $F$ due to the lemma by Yamaguchi.) More generally,
when $F$ is locally convex in $X$, for any $x\in F$ there is a neighborhood $U_x$ of $x$ such that
if $\dist_q|_F$ with $q\not\in F$ attains a minimum at $p\in F^\circ\cap U_x$, then the corresponding (\ref{eqn0.3})
holds for all $r\in F\cap U_x\setminus\{p\}$. Note that `extremal' has no local version because, to (\ref{eqn0.2}),
`$\forall\ r\in X\setminus\{p\}$' is equivalent to `$\forall\ r\in U\setminus\{p\}$' for any neighborhood $U$ of $p$.

Inspired by (\ref{eqn0.2}) and  (\ref{eqn0.3}), we introduce the following concept in the present paper:

\vskip2mm

\noindent{\bf Definition A.} In an $X\in \text{Alex}(k)$, a closed subset $F$ is called to be {\it quasi-convex}
if the following condition is satisfied: if the distance function $\dist_q|_F$ with $q\not\in F$
attains a minimum at $p\in F$, then
\begin{equation} \label{eqn0.4}
\forall\  r\in F\setminus\{p\},\ 	|\uparrow_p^q\uparrow_p^r|\leqslant\frac{\pi}{2} \text{ for some $[pq]$ and some $[pr]$}.
\end{equation}
More generally, a subset $F$ (might not be closed) is called to be {\it locally quasi-convex} in $X$ if for any $x\in F$ there is a neighborhood $U_x$ of $x$ such that $F\cap U_x$ is closed in $X$ and  if $\dist_q|_F$ with $q\not\in F$
attains a minimum at $p\in F\cap U_x$, then the corresponding (\ref{eqn0.4}) holds for all $r\in F\cap U_x\setminus\{p\}$.

\vskip2mm

In the definition of `quasi-convex',
$F$ has to be closed; otherwise, any open subset will be quasi-convex. And for the sake of completeness, we make a convention that both the empty set and a single point are quasi-convex,
and thus a set of isolated points is locally quasi-convex.

It is obvious that quasi-convex subsets include not only all closed convex subsets without boundary but also all extremal subsets (for other examples please see Section 2 below). And it turns out that a closed and locally quasi-convex
subset in a Riemannian manifold must be a totally geodesic submanifold (see Remark \ref{rem0.4} below).

We would like to point out that, in Definition A, (\ref{eqn0.4}) cannot be replaced by the corresponding (\ref{eqn0.1}) for all $p_i\in F$ with $p_i\to p$ (otherwise, any submanifold in a Riemannian manifold will be quasi-convex).

To $p,q,r\in X$, we associate $\tilde p, \tilde q, \tilde r\in \Bbb S^2_k$
with $|\tilde p\tilde q|=|pq|, |\tilde p\tilde r|=|pr|$ and $|\tilde q\tilde r|=|qr|$ \footnote{For $k>0$,
if $|pq|+|pr|+|qr|=\frac{2\pi}{\sqrt k}$ and $|pq|=\frac{\pi}{\sqrt k}$, it should be added that $\tilde p, \tilde q$ and $\tilde r$ lie in a geodesic of length $\frac{\pi}{\sqrt k}$.},
where $\Bbb S^2_k$ denotes the complete and simply connected space form of dimension 2 and curvature $k$;
and we denote by $\tilde\angle_k qpr$ the angle at $\tilde p$ of the triangle
$\triangle \tilde p\tilde q\tilde r$. By Toponogov's Theorem (see Theorem \ref{thm1.1} below),
it is obvious that condition (\ref{eqn0.4}) implies `$\tilde\angle_k qpr\leqslant\frac\pi2$', but not vice versa.
However, an easy but not obvious observation is that they are equivalent in the situation of Definition A.

\vskip2mm

\noindent {\bf Proposition B.}{\it \label{proB}
In an $X\in \text{\rm Alex}(k)$, a closed subset $F$ is quasi-convex
if and only if the following holds: if the distance function $\dist_q|_F$ with  $q\not\in F$ attains a minimum at $p\in F$, then
$$ \tilde\angle_k qpr\leqslant\frac{\pi}{2} \text{ for all }  r\in F\setminus\{p\}.$$}
\hskip5mm
Our main result aims to show that quasi-convex subsets have the following similar properties as convex and extremal subsets.

\vskip2mm

\noindent {\bf Theorem C.}{\it \label{thm1}
Let $F$ be a locally quasi-convex subset in an $X\in \text{\rm Alex}(k)$. Then the following hold:

\vskip1mm

\noindent{\rm (C1)} For any $p\in F$,  $\Sigma_p F$ is quasi-convex in $\Sigma_p X $, where
$$
\Sigma_pF\triangleq\{\xi\in\Sigma_pX|\ \text{there is $p_i\in F$ and $[pp_i]\ (\subset X)$ such that $p_i\to p$ and $\uparrow_p^{p_i}\to \xi$ as $i\to\infty$}\}.
$$

\noindent{\rm (C2)} A shortest path in $F$ is a quasi-geodesic in $X$.

\vskip1mm

\noindent{\rm (C3)} In addition, we suppose that $F$ is closed in $X$ and $k=1$. If $F$ contains no isolated point, or if $F$ is quasi-convex in $X$ and contains at least two points,
then $|qF|\leqslant \frac\pi2$ for any $q\in X$ (where $|qF|$ is the distance between $q$ and $F$) and the equality
implies that $|qp|=\frac\pi2$ for all $p\in F$.     }

\begin{remark}\label{rem0.2}{\rm (0.2.1) When $F$ is locally convex in $X$, it is well known that $\Sigma_p F$ with $p\in F$ is convex in $\Sigma_p X$
and that a shortest path in $F$ is a geodesic in $X$, i.e. the corresponding (C1) and (C2) hold. And if $F$ has no boundary, then the corresponding (C3) is the Lemma 2.5 in \cite{Ya}.

\vskip1mm

\noindent(0.2.2) When $F$ is extremal in $X$, due to \cite{PP1}, the corresponding (C1)-(C3) hold, and (C2) is called to be a generalized Liberman theorem.
Note that if $p$ is an isolated point of $F$, then $\Sigma_pF$ is empty and (\ref{eqn0.2}) implies that $\diam(\Sigma_p X)$
(i.e. the diameter of $\Sigma_p X$) is less than or equal to $\frac\pi2$.
And if $\Sigma_pF$ consists of a single point $\hat p$, then (\ref{eqn0.2}) implies that $\Sigma_pX$ belongs to the closure of the ball $B_{\hat p}(\frac\pi2)$.

\vskip1mm

\noindent(0.2.3) Due to the fact that (\ref{eqn0.2}) is much stronger than (\ref{eqn0.3}) and (\ref{eqn0.4}),
one can prove that a closed subset $F$ containing at least two points is extremal in $X$
if $\Sigma_p F$ is extremal in $\Sigma_p X$ for any $p\in F$ with
an extra requirement: $\diam(\Sigma_pX)\leqslant\frac{\pi}{2}$ or $\Sigma_pX\subseteq\overline{B_{\hat p}(\frac{\pi}{2})}$ respectively if $\Sigma_pF=\emptyset$ or $\{\hat p\}$ (\cite{PP1}). However,
there is no such strong conclusion on quasi-convex subsets. }
\end{remark}

\begin{remark}\label{rem0.3}{\rm The quasi-geodesic in (C2) was defined by
Petrunin and Perel'man-Petrunin in entirely different ways (ref. \cite{PP1}), and studied intensively in \cite{PP2}.
A quasi-geodesic must be a geodesic in a Riemannian manifold, but might not
in a general space of $\text{\rm Alex}(k)$ (\cite{PP2}).
Because we can substitute quasi-geodesics for geodesics when applying Toponogov's Theorem in some situations (\cite{PP2}, \cite{Pet}), quasi-geodesics play an important role in studying Alexandrov spaces.}
\end{remark}

\begin{remark}\label{rem0.4}{\rm In a complete Riemannian manifold, since a quasi-geodesic is a geodesic, due to (C2) it is
obvious that a closed and locally quasi-convex subset must be a totally geodesic submanifold. On the other hand, locally quasi-convex subsets also include all extremal subsets. Moreover, there are so many examples of locally quasi-convex subsets which are not locally convex nor extremal (see Section 2). Hence, locally quasi-convex subsets illustrate both similarities and differences between Riemannian manifolds and general Alxandrov spaces, and thus are worthy of further study.}
\end{remark}

\begin{remark}\label{rem0.5}{\rm The rough idea of our proofs, especially for (C1) and (C2), is similar to that of the corresponding proofs in \cite{PP1}.
However, it is obvious that (\ref{eqn0.2}) is much stronger than (\ref{eqn0.4}), so our proof will face much more (essential) difficulties. In order to overcome them, not only much more careful calculations should be needed (such as Lemma \ref{lem3.2} and (3.11)), but also
the geometries of (\ref{eqn0.4}) should be explored much more sufficiently (such as Corollary \ref{cor3.5}, Lemmas \ref{lem4.3} and \ref{lem4.4}). }
\end{remark}

We end this section with an important further question on quasi-convex subsets.
Similar to Boundary Conjecture in Alexandrov geometry
and the conjecture for extremal subsets (\cite{PP2}), we can ask whether a primitive component of a quasi-convex subset of codimension 1 or 2 is an Alexandrov space with lower curvature bound (refer to \cite{PP1} for `primitive', and see Remark \ref{rem3.4} below for the dimension of a quasi-convex subset).

%\noindent (0.8.2) A nice property of an extremal subset $F$ is that two sufficiently close points in $F$ can be jointed
%with a curve in $F$ (\cite{PP1}). A natural question is how about quasi-convex subsets.

\vskip2mm

The rest of the paper is organized as follows. In Section 1, we will give a proof for Proposition B.
In Section 2, we provide some examples of locally quasi-convex subsets. In Section 3, we will prove (C1) of Theorem C.
In Section 4, we will give a proof and some applications of (C3). Based on Sections 3 and 4,
we will give a proof for (C2) in Section 5.

%%%%%%%%%%%%%%%%%%%%%%%%%%%%%%%%%%%%%% section 0+ %%%%%%%%%%%%%%%%%%%%%%%%%%%%%%%%%%%%%%%%%%%%%%%%%%

\section{Proof of Proposition B}

Before giving the proof, we first present a brief version of Toponogov's Theorem and
the first variation formula in Alexandrov geometry with lower curvature bound.

\begin{theorem}[\cite{BGP}]\label{thm1.1} In an $X\in \text{\rm Alex}(k)$, for any
minimal geodesics $[pq]$ and $[pr]$, we have that
$$|\uparrow_p^q\uparrow_p^r|\geqslant\tilde\angle_kqpr.$$
\end{theorem}

Refer to the content right above Proposition B for the notation `$\tilde\angle_kqpr$'.
And in the rest of the paper, we will denote by $\Uparrow_p^r$ all directions from $p$ to $r$.

\begin{lemma}[\cite{BGP}, \cite{AKP}]\label{lem1.2} In an $X\in \text{\rm Alex}(k)$,
if $[pq]$ and $[pr]$ satisfy $|\uparrow_p^{q}\uparrow_p^{r}|=|\uparrow_p^{q}\Uparrow_p^{r}|$,
then for $q'\in[pq]$ converging to $p$ we have that
$|q'r|=|pr|-|pq'|\cos|\uparrow_p^{q}\uparrow_p^{r}|+o(|pq'|)$,
which implies that $\lim\limits_{q'\to p}\tilde\angle_kq'pr=|\uparrow_p^{q}\uparrow_p^{r}|$.
\end{lemma}

\begin{proof}[Proof of Proposition B]\

By Theorem \ref{thm1.1}, it is obvious that the quasi-convexity of $F$ (see (\ref{eqn0.4})) implies $\tilde\angle_k qpr\leqslant\frac{\pi}{2}$. Namely, we need only to prove Proposition B for the sufficiency.
If $F$ is not quasi-convex, then there is $q\not\in F$, $p\in F$ and $r\in F\setminus\{p\}$ such that $\dist_q|_F$
attains a minimum at $p$ and meantime $|\uparrow_p^{q}\uparrow_p^{r}|>\frac\pi2$
for any $[pq]$ and $[pr]$. Without loss of generality, we can select $[pq]$ and $[pr]$ satisfying  $|\uparrow_p^{q}\uparrow_p^{r}|=|\uparrow_p^{q}\Uparrow_p^{r}|$.
Then by Lemma \ref{lem1.2} and because $|\uparrow_p^{q}\uparrow_p^{r}|>\frac\pi2$,
for $q'\in [qp]$ sufficiently close to $p$, we have that
$$\tilde\angle_kq'pr>\frac\pi2.$$
On the other hand, note that $q'\not\in F$ and $\dist_{q'}|_F$ attains the minimum at $p$,
so by our assumption we have that $\tilde\angle_kq'pr\leqslant\frac\pi2$, a contradiction.
\end{proof}

Similarly, from the idea of considering $q'\in [qp]$ close to $p$, we can see the following proposition.

\begin{prop}\label{pro1.3}
 Let $F$ be a quasi-convex subset in an $X\in \text{Alex}(k)$, and let $q\in X$. If $\dist_q|_F$
attains a LOCAL minimum at $p\in F$ with $p\neq q$, then
$$\forall\  r\in F\setminus\{p\},\ 	|\uparrow_p^q\uparrow_p^r|\leqslant\frac{\pi}{2} \text{ for any $[pq]$ and some $[pr]$}.$$
\end{prop}

\begin{proof}
Note that $q'\in [pq]$ sufficiently close to $p$ does not belong to $F$, and $\dist_{q'}|_F$ attains the minimum at $p$.
It is clear that $\uparrow_{p}^q=\uparrow_{p}^{q'}$, and that there is a unique minimal geodesic between $p$ and $q'$.
Hence, by the quasi-convexity of $F$ (see (\ref{eqn0.4})),
there is some $[pr]$ such that $|\uparrow_{p}^q\uparrow_{p}^{r}|=|\uparrow_{p}^{q'}\uparrow_{p}^{r}|\leqslant\frac\pi2$.
\end{proof}

%%%%%%%%%%%%%%%%%%%%%%%%%%%%%%%%%%%%%% section 1 %%%%%%%%%%%%%%%%%%%%%%%%%%%%%%%%%%%%%%%%%%%%%%%%%%

\section{Examples}

This section aims to supply some examples of locally quasi-convex subsets.

\begin{example}\label{ex2.1}{\rm\

\noindent{\bf(2.1.1)} From Remark \ref{rem0.4}, we know that closed and locally quasi-convex subsets coincide with totally geodesic submanifolds in a complete Riemannian manifold.
However, a totally geodesic submanifold might not be quasi-convex (e.g., a cylindrical spiral in a cylinder is not quasi-convex).

\vskip2mm

\noindent {\bf(2.1.2)} In a general $X\in \text{\rm Alex}(k)$, a locally convex subset without boundary must be locally quasi-convex, but a locally convex subset with non-empty boundary might not be locally quasi-convex.
	
\vskip2mm

\noindent {\bf(2.1.3)} In a unit sphere $\Bbb S^n$,  the set of two antipodal points is quasi-convex.
In general, a set of isolated points in $X\in \text{\rm Alex}(k)$ might not be quasi-convex although it is locally quasi-convex.
	
\vskip2mm

\noindent {\bf(2.1.4)} Let $X\in \text{\rm Alex}(k)$ with $\partial X\neq\emptyset$. The doubling of $X$, denoted by $D(X)$, also belongs to $\text{\rm Alex}(k)$ (\cite{Pe2}). Note that
$\partial X$ is extremal in $X$ (\cite{PP1}), so is quasi-convex. However, as a subset of $D(X)$,
$\partial X$ is not extremal, but still quasi-convex.

\vskip2mm

\noindent {\bf(2.1.5)} Let $X_i\in \text{\rm Alex}(k)$ with $\partial X_i\neq\emptyset$, $i=1,2$. If $\partial X_1$ is isometric to $\partial X_2$ with respect to their induced intrinsic metrics from $X_i$, then
the gluing space of $X_1$ and $X_2$ along $\partial X_i$ by the isometry, denoted by $X_1\cup_{\partial X_i}X_2$, also
belongs to $\text{\rm Alex}(k)$ (\cite{Pet}).
In general, $\partial X_i$ as a subset of $X_1\cup_{\partial X_i}X_2$ is not locally quasi-convex unless each $\partial X_i$ is convex in $X_i$.
For example, if $X_1=S^1\times [0,+\infty)$ and $X_2$ is a disc bounded by the circle $S^1$, the $S^1$ is not locally quasi-convex in the `barrel' $X_1\cup_{S^1\times\{0\}}X_2$.

\vskip2mm

\noindent {\bf(2.1.6)} Inspired by the `barrel' in (2.1.5), one can consider a solid cylinder, i.e. $D\times [0,+\infty)$ with $D$ being a disc bounded by a circle $S^1$, which belongs to $\text{\rm Alex}(0)$.
Note that $S^1\times \{0\}$ is extremal (but not locally convex) in $D\times [0,+\infty)$.
If we glue a solid cone whose bottom is isometric to $D$ on to $D\times [0,+\infty)$ along $D\times \{0\}$, then
$S^1\times \{0\}$ is quasi-convex in the gluing space (whose curvature is also $\geqslant 0$), but not locally convex nor extremal.

\vskip2mm

\noindent{\bf(2.1.7)}  Let $\Sigma\in \text{\rm Alex}(1)$. One can define a cone over $\Sigma$, denoted by $C(\Sigma)$, which
belongs to $\text{\rm Alex}(0)$ (refer to \cite{BGP} for details). For example, if $\Sigma$ is  a circle with perimeter $\leqslant 2\pi$, then $C(\Sigma)$ is isometric to a circular cone in $\Bbb R^3$. It is not hard to see that if $F$ is quasi-convex in $\Sigma$ and contains at least two points, then $C(F)$ is quasi-convex in $C(\Sigma)$
(this might not be true for `locally quasi-convex').

\vskip2mm

\noindent{\bf(2.1.8)}  Let $\Sigma_i\in \text{\rm Alex}(1)$, $i=1,2$.
One can define a join of $\Sigma_1$ and $\Sigma_2$, denoted by $\Sigma_1*\Sigma_2$, which also
belongs to $\text{\rm Alex}(1)$ (refer to \cite{BGP} for details).
In particular, if $\Sigma_1$ consists of two points with distance $\pi$, then $\Sigma_1*\Sigma_2$
is called to be a (spherical) suspension over $\Sigma_2$, denoted by $S(\Sigma_2)$.
For example, If $\Sigma_i$ is the $m_i$-dimensional unit sphere, then $\Sigma_1*\Sigma_2$ is the $(m_1+m_2+1)$-dimensional unit sphere; in particular, $S(\Sigma_2)$ is the  $(m_2+1)$-dimensional unit sphere.
Note that each $\Sigma_i$ as a subset in $\Sigma_1*\Sigma_2$ is convex; and as an example of extremal subsets,
\begin{equation} \label{eqn1.1}
\text{$\Sigma_1$ is extremal in $\Sigma_1*\Sigma_2$ if and only if the diameter $\diam(\Sigma_2)\leqslant\frac\pi2$.}
\end{equation}
Moreover, a quasi-convex subset $F_i$ in $\Sigma_i$
is also quasi-convex in $\Sigma_1*\Sigma_2$, but in general $F_1*F_2$ is not quasi-convex in $\Sigma_1*\Sigma_2$ (for example, each $\Sigma_i$ is a circle of
perimeter $<2\pi$, and each $F_i$ consists of two antipodal points in $\Sigma_i$). However,
for the case where $\Sigma_1*\Sigma_2$ is equal to $S(\Sigma_2)$, if $F_1=\Sigma_1$ and $F_2$ is quasi-convex in $\Sigma_2$ and contains at least two points, then $F_1*F_2$ is quasi-convex in $S(\Sigma_2)$.

\vskip2mm

\noindent{\bf(2.1.9)} As a special case of (2.1.8), consider $S(\Sigma_2)$ ($=\Sigma_1*\Sigma_2$) with $\Sigma_2$ being a sphere of diameter $\leqslant\pi$. Put $F_1=\Sigma_1$, and let $F_2$ be two antipodal points in $\Sigma_2$.
Note that $F_1$ is quasi-convex in $S(\Sigma_2)$ (and is extremal if $\diam(\Sigma_2)\leqslant\frac\pi2$); and $F_2$ is quasi-convex in $\Sigma_2$ and in $S(\Sigma_2)$.
And $F_1*F_2$, a union of two antipodal longitudes, is quasi-convex in $S(\Sigma_2)$, but not locally convex if $\diam(\Sigma_2)<\pi$ nor extremal even
if $\diam(\Sigma_2)\leqslant\frac\pi2$.		
}	
\end{example}

In (2.1.9), if $\Sigma_2$ is an arbitrary space in $\text{\rm Alex}(1)$, and if a quasi-convex subset $F$ in $S(\Sigma_2)$
contains at least two points including $z_1$, then we have a necessary condition for $F$ as follows.

\begin{prop} \label{prop2.2}
Let $Z=\{z_1,z_2\}*Y$ with $|z_1z_2|=\pi$ and $Y\in \text{\rm Alex}(1)$,
and let $F$ be a quasi-convex subset in $Z$. If $\{z_1\}\subsetneq F$,
then $F=\{z_1,z_2\}*(F\cap Y)$ with $F\cap Y$ being quasi-convex in $Y$.
\end{prop}

Note that in Proposition \ref{prop2.2}, `$F=\{z_1,z_2\}*(F\cap Y)$' can be written as `$F=\{z_1,z_2\}*\Sigma_{z_1}F$'.

\begin{proof} First of all, we show that $z_2 \in F$.
If it is not true, then there is $p\in F\setminus\{z_1, z_2\}$ such that $|z_2p|=|z_2F|$ because $\{z_1\}\subsetneq F$. However, by the suspension structure of $Z$,
$$
\text{for any $x\in Z\setminus\{z_1, z_2\}$, there is a unique $[z_1z_2]$, denoted by $[z_1z_2]_x$, such that $x\in [z_1z_2]_x$.}
$$
Note that $[z_ip]\subsetneq [z_1z_2]_p$  ($i=1, 2$) is the unique minimal geodesic between $p$ and $z_i$ with $|\uparrow_p^{z_1}\uparrow_p^{z_2}|=\pi$,
which contradicts the quasi-convexity of $F$ (i.e. (\ref{eqn0.4})).

Next, if $F\neq\{z_1,z_2\}$, it suffices to show that $[z_1z_2]_q\subset F$ for any $q\in F\setminus\{z_1,z_2\}$,
which implies that $F=\{z_1,z_2\}*(F\cap Y)$
and $F\cap Y$ is quasi-convex in $Y$ (cf. (2.1.8)).

If $[z_1z_2]_q\not\subset F$, without loss of generality, we can assume that $[qz_1]\not\subset F$,
and that there is $[qq_0]\subsetneq [qz_1]$ such that $[qq_0]\cap F=\{q\}$.
Let $q_i\in [qq_0]\setminus\{q\}$ with $q_i\to q$ as $i\to\infty$,
and let $p_i\in F$ such that $|q_ip_i|=|q_iF|$.
We will show that there is large $i_0$ such that, for $i\gg i_0$,
\begin{equation}\label{eqn2.0}
\tilde\angle_1q_ip_ip_{i_0}>\frac\pi2,
\end{equation}
which contradicts the quasi-convexity of $F$ by Proposition B (and the proof is finished).

Observe that $p_i\notin [z_1z_2]_q$ for large $i$ (otherwise, for $j=1$ or 2,
`$|\uparrow_{p_i}^{q_i}\uparrow_{p_i}^{z_j}|=\pi$' will contradict the quasi-convexity of $F$).
Then we can consider a `lune' ${\cal L}_i$ bounded by $[z_1z_2]_q$ and $[z_1z_2]_{p_i}$.
Note that there is $[yy_i]\subset Y$ with $y\in[z_1z_2]_q$ and $y_i\in [z_1z_2]_{p_i}$
such that ${\cal L}_i=\{z_1,z_2\}*[yy_i]$. Passing to a subsequence,
we can assume that $\uparrow_{y}^{y_i}$ converges as $i\to\infty$;
and thus for the $[qp_i]$ in ${\cal L}_i$, $\uparrow_{q}^{p_i}$ converges to some $\xi\in\Sigma_qZ$ as $i\to\infty$.
Moreover, we have that
\begin{equation}\label{eqn2.1}
|\uparrow_{q}^{z_1}\xi|<\frac\pi2.
\end{equation}
In fact, by the quasi-convexity of $F$, one can conclude that $|\uparrow_{p_i}^{q_i}\uparrow_{p_i}^{z_j}|=\frac{\pi}{2}$ ($j=1,2$) for the $[q_ip_i]$ in ${\cal L}_i$. And thus it follows that
\begin{equation}\label{eqn2.3}
|\uparrow_{q_i}^{p_i}\uparrow_{q_i}^{z_j}|\to\frac\pi2 \text{ as $i\to\infty$}.
\end{equation}
This together with `$|q_ip_i|<|q_iq|\to 0$' and `$|\uparrow_{q}^{z_1}\uparrow_{q}^{p_i}|\to|\uparrow_{q}^{z_1}\xi|$ as $i\to\infty$' implies (\ref{eqn2.1}).

Furthermore, based on (\ref{eqn2.1}) we claim that, for sufficiently small $\epsilon>0$, there is large $i_0$ such that
\begin{equation}\label{eqn2.5}
\left||\uparrow_{q_i}^{p_i}\uparrow_{q_i}^{p_{i_0}}|-\left(\frac\pi2-|\uparrow_{q}^{z_1}\xi|\right)\right|<3\epsilon \text{ for $i\gg i_0$}.
\end{equation}
By Theorem \ref{thm1.1}, (\ref{eqn2.5}) together with $ |q_ip_i|\ll |p_ip_{i_0}|$ for $i\gg i_0$ (note that $|q_ip_i|<|q_iq|\to 0$) implies that $\tilde\angle_1p_iq_ip_{i_0}+\tilde\angle_1q_ip_{i_0}p_i<\frac\pi2$.
It follows that $\tilde\angle_1q_ip_ip_{i_0}>\frac\pi2$, i.e. (\ref{eqn2.0}) holds.
Hence, we need only to verify the claim right above.
Note that, for sufficiently small $\epsilon$, there is large $i_0$ such that
$$|\uparrow_{q}^{p_i}\uparrow_{q}^{p_{i_0}}|<\epsilon \text{ and } |qp_i|\ll |qp_{i_0}| \text{ for } i\gg i_0.$$
This together with $q_i\to q$ implies that  $\left||\uparrow_{q_i}^{p_{i_0}}\uparrow_{q_i}^{z_1}|-|\uparrow_{q}^{z_1}\xi|\right|<2\epsilon$ for $i\gg i_0$.
And note that ${\cal L}_i$ satisfies that $\uparrow_{y}^{y_i}$ converges as $i\to\infty$, by which and (\ref{eqn2.3}) we can assume that
$\left||\uparrow_{q_i}^{p_i}\uparrow_{q_i}^{p_{i_0}}|+|\uparrow_{q_i}^{p_{i_0}}\uparrow_{q_i}^{z_1}|-\frac\pi2\right|<\epsilon$ for $i\gg i_0$. Then (\ref{eqn2.5}) follows, i.e. the claim is verified.
\end{proof}

For another important kind of examples of quasi-convex subsets, we have the following proposition.

\begin{prop}\label{prop2.3}
Let $X\in \text{\rm Alex}(k)$, and let $\Gamma$ be a compact group which acts on $X$ by isometries
with non-empty fixed point set $F$. Then $F$ is quasi-convex in $X$.
\end{prop}

\begin{proof} Note that $F$ is a closed subset in $X$. Let $q$ be an arbitrary point in $X\setminus F$, and assume that $\dist_q|_F$ attains a minimum at $p\in F$.
We need to verify that for any $r\in F\setminus\{p\}$ there is some $[pq]$ and $[pr]$ such that $|\uparrow_p^q\uparrow_p^r|\leqslant\frac\pi2$. Consider the quotient map $\tau: X\to X/\Gamma$. Note that the orbit space $X/\Gamma$ endowed with the induced metric from $X$ also belongs to $\text{\rm Alex}(k)$ ([BGP]). Since $F$ is the fixed point set of the isometric $\Gamma$-action,
for any $[\tau(p)\tau(q)]$ and $[\tau(p)\tau(r)]$ there is $[pq]$ and $[pr]$  such that $\tau([pq])=[\tau(p)\tau(q)]$, $\tau([pr])=[\tau(p)\tau(r)]$  and
$$|\uparrow_p^q\uparrow_p^r|=|\uparrow_{\tau(p)}^{\tau(q)}\uparrow_{\tau(p)}^{\tau(r)}|;$$
moreover, $\dist_{\tau(q)}|_{\tau(F)}$ attains a minimum at $\tau(p)\in \tau(F)$ on $X/\Gamma$.
Meantime, it is known that
\begin{equation}\label{eqn2.6}
\text{$\tau(F)$ is extremal in $X/\Gamma$ (\cite{PP1})}.
\end{equation}
So, $|\uparrow_{\tau(p)}^{\tau(q)}\uparrow_{\tau(p)}^{\tau(r)}|\leqslant\frac\pi2$ by (\ref{eqn0.2}), and thus
$|\uparrow_p^q\uparrow_p^r|\leqslant\frac\pi2$.
\end{proof}

\begin{remark}\label{rem2.4}{\rm We would like to point out that although \cite{PP1} shows that the $F$
in Proposition \ref{prop2.3} is extremal as a subset of $X/\Gamma$, it does not describe what kind of subset
$F$ is in $X$.}
\end{remark}

%%%%%%%%%%%%%%%%%%%%%%%%%%%%%%%%%%%%%% section 3 %%%%%%%%%%%%%%%%%%%%%%%%%%%%%%%%%%%%%%%%%%%%%%%%%%

\section{Proof for (C1) of Theorem C}

In this section, the main goal is to prove (C1), i.e.,
$\Sigma_p F$ is quasi-convex in $\Sigma_p X $ if $F$ is locally quasi-convex
in $X$.  As mentioned in Remark \ref{rem0.5}, the proof will be more difficult than that for $\Sigma_p F$ being extremal in $\Sigma_p X$ if $F$ is extremal in $X$ (see Remark \ref{rem3.0} below).

\begin{lemma}[\cite{BGP}]\label{lem3.1}
Let $X\in \text{\rm Alex}(k)$, and $p\in X$. Then for any small $\epsilon>0$
there is a neighborhood $U_\epsilon$ of $p$ such that, for any $[pq],[pr]\subset U_\epsilon$ with $q,r\neq p$,
$$0\leqslant|\uparrow_p^q\uparrow_p^r|-\tilde\angle_k qpr<\epsilon.$$
\end{lemma}

\begin{lemma}\label{lem3.2}
Let $p_i, q_i, r_i$ be points in a metric space with $|p_ir_i|\to 0$ as $i\to\infty$,
and let $k$ be a real number. If $\lim\limits_{i\to\infty}\tilde\angle_{k} q_ip_ir_i$ exists with $\limsup\limits_{i\to\infty}|q_ip_i|<\frac{\pi}{\sqrt{k}}$ for $k>0$,
then, for any $k'\neq k$ with $\limsup\limits_{i\to\infty}|q_ip_i|<\frac{\pi}{\sqrt{k'}}$ for $k'>0$, we have that $\lim\limits_{i\to\infty}\tilde\angle_{k'} q_ip_ir_i=\lim\limits_{i\to\infty}\tilde\angle_{k} q_ip_ir_i$.
\end{lemma}

The lemma should be a known fact to experts because
the calculation is only on $\Bbb S_k^2$ and $\Bbb S_{k'}^2$, so we omit its proof.
For the case where both $|p_ir_i|$ and $|p_iq_i|$ converge to 0 as $i\to\infty$ in Lemma \ref{lem3.2},
the conclusion is used as a fact in [BGP] to define angles in Alexandrov spaces.

\begin{proof}[Proof for (C1) of Theorem C]\
	
Let $F$ be a locally quasi-convex subset in $X\ (\in\text{\rm Alex}(k))$, and $p\in F$.
If $\Sigma_p F$ is not quasi-convex in $\Sigma_p X\ (\in\text{\rm Alex}(1))$, we will derive a contradiction to the local quasi-convexity of $F$.

Note that if $\Sigma_p F$ is not quasi-convex in $\Sigma_p X$, then there exists an $\eta'\in \Sigma_p X\setminus \Sigma_p F$ such that $\dist_{\eta'}|_{\Sigma_p F}$ attains a
minimum at $\xi\in \Sigma_p F$ and there is $\zeta\in \Sigma_pF$ satisfying
\begin{equation} \label{eqn3.1}
|\Uparrow_{\xi}^{\eta'}\Uparrow_{\xi}^{\zeta}|>\frac\pi2
\end{equation}
(refer to Lemma \ref{lem1.2} for the notation `$\Uparrow$'). Note that it has to hold that
\begin{equation} \label{eqn3.2}
|\xi\zeta|<\pi
\end{equation}
(if $|\xi\zeta|=\pi$, then $\Sigma_pX$ is a suspension $\{\xi,\zeta\}*Y$ with $Y\in\text{\rm Alex}(1)$,
which implies that $\Uparrow_{\xi}^{\zeta}=\Sigma_{\xi}(\Sigma_pX)$, a contradiction to (\ref{eqn3.1})).
Take an arbitrary $[\eta'\xi]$, and let $\eta\in [\eta'\xi]^\circ$. Note that
\begin{equation} \label{eqn3.3}
\text{$\xi$ is the unique minimum point of $\dist_{\eta}|_{\Sigma_p F}$},
\end{equation}
and
$|\Uparrow_{\xi}^{\eta}\Uparrow_{\xi}^{\zeta}|=|\uparrow_{\xi}^{\eta}\Uparrow_{\xi}^{\zeta}|
\geqslant|\Uparrow_{\xi}^{\eta'}\Uparrow_{\xi}^{\zeta}|>\frac\pi2$.
And thus if we let
\begin{equation} \label{eqn3.4}
\text{$\eta$ be sufficiently close to $\xi$ (along $[\eta'\xi]$)},
\end{equation}
then by Lemma \ref{lem1.2} we can assume that
\begin{equation} \label{eqn3.5}
\tilde\angle_1\eta\xi\zeta>\frac\pi2.
\end{equation}
Consider $\{p_i\}_{i=1}^\infty\subset F$ and $\{r_i\}_{i=1}^\infty\subset F$ with $p_i,r_i\to p$,
$\uparrow_{p}^{p_i}\to \xi$ and $\uparrow_{p}^{r_i}\to \zeta$ as $i\to\infty$ and
\begin{equation} \label{eqn3.6}
|pr_i|\geqslant|pp_i|.
\end{equation}
Furthermore, we can select $q_i\in X\setminus F$ such that $\uparrow_{p}^{q_i}\to \eta$ as $i\to\infty$ (note that $\eta\notin\Sigma_pF$) and
\begin{equation} \label{eqn3.7}
|pq_i|\cdot \cos |\eta\xi|=|pp_i|.
\end{equation}
By (\ref{eqn3.4}) and (\ref{eqn3.7}), $q_i$ also converges to $p$ as $i\to\infty$. Since there is a neighborhood $U$ of $p$ such that $F\cap U$ is closed in $X$ (by the local quasi-convexity of $F$),
for large $i$ there is $s_i\in F\cap U$ such that $|q_is_i|=|q_iF|$.
Note that $s_i$ also converges to $p$ as $i\to\infty$ because $|q_is_i|\leqslant |q_ip|$.
We claim that $|\Uparrow_{s_i}^{q_i}\Uparrow_{s_i}^{r_i}|>\frac\pi2$ for large $i$.
This obviously contradicts the local quasi-convexity of $F$, so it suffices to verify the claim in the rest of the proof.

First of all, passing to a subsequence we can assume that $\uparrow_p^{s_i}\to\xi'\in\Sigma_pF$ as $i\to\infty$;
and we will show that
\begin{equation} \label{eqn3.8}
\xi'=\xi,\ \lim_{i\to\infty}\frac{|ps_i|}{|pq_i|}=\cos|\eta\xi|\text{ and } \lim_{i\to\infty}\frac{|ps_i|}{|pp_i|}=1.
\end{equation}
A known fact is that $(\frac{1}{|pq_i|}X,p)\to C_pX$ as $i\to\infty$ (\cite{BGP}), where $\frac{1}{|pq_i|}X$ denotes the $X$ endowed with a metric $\frac{1}{|pq_i|}$
times than the original one of $X$. It then follows that
$$
\lim_{i\to\infty}\frac{|q_ip_i|}{|pq_i|}=\sin|\eta\xi| \text{ and } \lim_{i\to\infty}\frac{|q_is_i|}{|pq_i|}=\sin |\eta\xi'|,
$$
and $|\eta\xi'|\leqslant\frac\pi2$ because $|q_is_i|\leqslant |pq_i|$.
These plus `$|\eta\xi|\leqslant |\eta\xi'|$ and $|q_ip_i|\geqslant |q_is_i|$' imply that
$\lim\limits_{i\to\infty}\frac{|q_is_i|}{|q_ip_i|}=1$ and $|\eta\xi|=|\eta\xi'|$,
and thus  (\ref{eqn3.8}) follows if we take into account (\ref{eqn3.3}) and (\ref{eqn3.7}).

Next, we will show $\tilde\angle_kq_is_ir_i>\frac\pi2$ for large $i$,
which implies that $|\Uparrow_{s_i}^{q_i}\Uparrow_{s_i}^{r_i}|>\frac\pi2$ by Theorem \ref{thm1.1},
i.e. the claim is true (and the proof is finished). For simpleness, we can just consider a special case where $r_i$ can be selected to satisfy $\lim\limits_{i\to\infty}\frac{|pp_i|}{|pr_i|}=0$ (cf. (\ref{eqn3.6})), which implies
\begin{equation}\label{eqn3.9}
\text{$\lim\limits_{i\to\infty}\frac{|ps_i|}{|pr_i|}=\lim\limits_{i\to\infty}\frac{|pq_i|}{|pr_i|}=0$\ \ and\ \
$\lim\limits_{i\to\infty}\frac{|pr_i|}{|r_is_i|}=1$}.
\end{equation}
By Lemma \ref{lem3.2}, it suffices to show that
$\lim\limits_{i\to\infty}\tilde\angle_0q_is_ir_i>\frac\pi2$, which is equivalent to
\begin{equation}\label{eqn3.10}
\lim\limits_{i\to\infty}\frac{|r_is_i|^2+|q_is_i|^2-|q_ir_i|^2}{2|r_is_i||q_is_i|}<0.
\end{equation}
By the Law of Cosine,
\begin{align*}
\hskip1cm & |r_is_i|^2+|q_is_i|^2-|q_ir_i|^2\\
\hskip1cm=&2|ps_i|^2-2|pr_i|\cdot|ps_i|\cos\tilde{\angle}_0 r_ips_i-2|ps_i|\cdot|pq_i|\cos \tilde{\angle}_0 s_i p q_i+2|pr_i|\cdot|pq_i|\cos \tilde{\angle}_0 r_i p q_i\\
\hskip1cm=&2|ps_i|\cdot|pq_i|\left(\frac{|ps_i|}{|pq_i|}-\cos \tilde{\angle}_0 s_i p q_i\right) +
  2|pr_i|\cdot|pq_i|\left(\cos \tilde{\angle}_0 r_i p q_i-\frac{|ps_i|}{|pq_i|}\cos \tilde{\angle}_0 r_i p s_i\right)\hskip1cm \text{(3.11)}\\
\triangleq & 2|ps_i|\cdot|pq_i|\cdot\text{I}+2|pr_i|\cdot|pq_i|\cdot\text{I\!I}.
\end{align*}
Note that, by Lemmas \ref{lem3.1} and \ref{lem3.2} together with the second term of (\ref{eqn3.8}), we have that
$$\lim_{i\to\infty}\text{I}=0 \text{ and } \lim_{i\to\infty}\text{I\!I}=\cos|\eta\zeta|-\cos|\eta\xi|\cos|\xi\zeta|=
\sin|\eta\xi|\sin|\xi\zeta|\cos\tilde\angle_1\eta\xi\zeta\eqno{(3.12)}$$
(for the last `$=$' notice that $\eta,\xi,\zeta\in\Sigma_pX\in\text{Alex}(1)$).
Putting (\ref{eqn3.9}), (3.11), (3.12) and
`$\lim\limits_{i\to\infty}\frac{|pq_i|}{|q_is_i|}=\frac{1}{\sin|\eta\xi|}$' together, we can see that
$$\lim\limits_{i\to\infty}\frac{|r_is_i|^2+|q_is_i|^2-|q_ir_i|^2}{2|r_is_i||q_is_i|}
=\sin|\xi\zeta|\cos\tilde\angle_1\eta\xi\zeta,$$ which implies (\ref{eqn3.10}) by (\ref{eqn3.5}).
\end{proof}	

We would like to point out that, in the proof right above, in order to show $\tilde\angle_kq_is_ir_i>\frac\pi2$ for large $i$ when $k\geqslant 0$, it suffices to show $\tilde\angle_0q_is_ir_i>\frac\pi2$
(i.e. $|r_is_i|^2+|q_is_i|^2-|q_ir_i|^2<0$) by Theorem \ref{thm1.1}.

\begin{remark}\label{rem3.0}{\rm In proving that $\Sigma_p F$ is extremal in $\Sigma_p X$
if $F$ is extremal in $X$, $\zeta$ is not necessary in $\Sigma_pF$
and can be close to $\xi$. Consequently, $\{r_i\}_{i=1}^\infty$ are not needed to
belong to $F$, and can be selected to satisfy $|pr_i|\cdot \cos |\zeta\xi|=|pp_i|$ (cf. (\ref{eqn3.7})).
Then from the fact that $(\frac{1}{|pq_i|}X,p)\to C_pX$ as $i\to\infty$, it will be easy to see `$\tilde\angle_kq_is_ir_i>\frac\pi2$' for large $i$,
which contradicts the extremal property of $F$ (\cite{PP1}).}
\end{remark}

\begin{remark}\label{rem3.4}{\rm
For a locally quasi-convex subset $F$, we can define a dimension by induction
due to (C1) of Theorem C (similar for an extremal subset in \cite{PP1}):
If $F$ is an empty set, then $\dim F$ is defined to be $-1$;
otherwise, it is defined to be $\max\limits_{p\in F} \{\dim \Sigma_p F\}+1$.}
\end{remark}

We end this section with a corollary from the proof for (C1) of Theorem C, which will play a key role in the proof for (C2) (see (\ref{eqn5.9}) below).

\begin{coro}\label{cor3.5} Let $F$ be a quasi-convex subset in $X\in\text{\rm Alex}(k)$, and let $p, r\in F$. Suppose that
there is $\eta'\in \Sigma_{p} X\setminus\Sigma_pF$ and $\xi\in \Sigma_pF$ satisfying $|\eta' \xi|=|\eta' \Sigma_pF|$.
If $|\bar\zeta \xi|=|\Uparrow_p^r \xi|$ for any $\bar\zeta\in \Uparrow_p^r$, then there is $\zeta\in \Uparrow_p^r$ such that
$$
|\Uparrow_{\xi}^{\eta'}\Uparrow_{\xi}^{\zeta}|\leqslant\frac\pi2 \ (\text{which implies } \tilde \angle_1 \eta' \xi \zeta\leqslant\frac{\pi}{2}).
$$
\end{coro}

\begin{proof} The proof is almost a copy of the proof for (C1) of Theorem C. The main difference is that
$r$ is a fixed point here (there is no sequence $\{r_i\}_{i=1}^\infty$, and we do not care whether $\bar\zeta\in\Sigma_pF$ or not for $\bar\zeta\in \Uparrow_p^r$).
Moreover, for the corresponding $\eta$ in the proof for (C1), the condition `$|\bar\zeta \xi|=|\Uparrow_p^r \xi|$ for any $\bar\zeta\in \Uparrow_p^r$' guarantees that there is $\zeta\in \Uparrow_p^r$ such that
$|\zeta \eta|=|\Uparrow_p^r \eta|$ and $|\zeta \xi|=|\Uparrow_p^r \xi|$.
As a result, we should apply Lemma \ref{lem3.6} below instead of Lemma \ref{lem3.1}
in estimating `$\cos\tilde{\angle}_0 rp q_i-\cos|\eta\zeta|$' and `$\cos\tilde{\angle}_0 rp s_i-\cos|\xi\zeta|$' (in order to show $\lim\limits_{i\to\infty}\tilde\angle_0q_is_ir>\frac\pi2$, see the `\text{I\!I}' in (3.11) and (3.12)).
\end{proof}	

\begin{lemma}[\cite{BGP}]\label{lem3.6}
Let $X\in\text{\rm Alex}(k)$, and let $q, p, r_i\in X$ with $r_i\to p$ as $i\to\infty$.
If there is $[pr_i]$ such that $\uparrow_p^{r_i}$ converges to $\xi\ (\in\Sigma_pX)$ as $i\to\infty$,
then
$$\lim_{i\to\infty}\frac{|qr_i|-|qp|}{|pr_i|}=-\cos|\Uparrow_p^{q}\xi|.$$
As a result, $\lim\limits_{i\to\infty}\tilde\angle_kqpr_i=|\Uparrow_p^{q}\xi|$.
\end{lemma}

Note that if $r_i\to p$ and some $\uparrow_p^{r_i}\to\xi$ as $i\to\infty$, then $\Uparrow_p^{r_i}\to\xi$ as $i\to\infty$ (\cite{BGP}). And Lemma \ref{lem1.2} is just Lemma \ref{lem3.6} in a special case where $r_i$ lies in a minimal geodesic for all large $i$.

%%%%%%%%%%%%%%%%%%%%%%%%%%%%%%%%%%%%%% section 4 %%%%%%%%%%%%%%%%%%%%%%%%%%%%%%%%%%%%%%%%%%%%%%%%%%

\section{Proof and applications of (C3)}

In this section, we will explore properties of quasi-convex subsets in Alexandrov spaces with curvature $\geqslant1$
including (C3) of Theorem C. Some of them will play an important role in the proof of (C2).

\subsection{Proof of (C3)}

In an $X\in \text{\rm Alex}(1)$, let $F$ be a closed and
locally quasi-convex subset containing no isolated point, or a quasi-convex subset containing at least two points.
We will prove that $|qF|\leqslant \frac\pi2$ for any $q\in X$, and the equality
implies that $|qp|=\frac\pi2$ for all $p\in F$ (i.e. (C3)).	

\begin{proof} We first consider the case where $F$ is closed and
locally quasi-convex and contains no isolated point. Note that we can assume that $F$ is connected (because
each component of $F$ is closed and locally quasi-convex).
Since $F$ is closed (and $X$ is compact, \cite{BGP}), there is $p_0\in F$ such that $|qp_0|=|qF|$.
By Definition A together with Proposition B, there is a neighborhood $U$ of $p_0$
such that $\tilde\angle_1 qp_0p \leqslant \frac{\pi}{2}$ for any $p\in F\cap U\setminus \{p_0\}$
(note that $F\cap U\setminus \{p_0\}\neq\emptyset$).
This plus `$|qp|\geqslant |qp_0|$' implies that $|qp_0|\leqslant\frac\pi2$, i.e. $|qF|\leqslant\frac\pi2$.
Moreover, if $|qp_0|=\frac\pi2$, it has to hold that $\tilde\angle_1 qp_0p=\frac\pi2$
and $|qp|=\frac\pi2$.
So, by the openness and closeness argument we can conclude that
$|qp|=\frac\pi2$ for all $p\in F$.

As for the case where $F$ is quasi-convex and contains at least two points, we just need to put $U=X$ in the proof right above.
\end{proof}

As a generalized version of (C3), we in fact can prove:

\begin{prop}\label{pro4.2}
Let $F$ be a locally quasi-convex subset in a $Y\in \text{\rm Alex}(1)$, and let $q\in Y$.
Suppose that $\dist_q|_F$ attains a local minimum
at $p_0\in F$. If $p_0$ is not an isolated point of $F$, then $|qp_0|\leqslant \frac\pi2$, and the equality
implies that there is a neighborhood $U$ of $p_0$ such that $|qp|=\frac\pi2$ for all $p\in F\cap U$.
\end{prop}

For the case where $q$ belongs to $F$ in Proposition \ref{pro4.2}, one can keep the following example in mind: $Y=\{z_1,z_2\}*S^1$ with $|z_1z_2|=\pi$ and $\diam(S^1)\leqslant\frac\pi2$
(cf. (2.1.9)), and $F=\{z_1,z_2\}*\{q,p_0\}$ with $q$ and $p_0$ being a couple of antipodal points of $S^1$.

\begin{proof} Note that we can assume that $p_0\neq q$. Then by the local quasi-convexity of $F$
together with Proposition \ref{pro1.3}, there is a neighborhood $U$ of $p_0$ such that
$|\Uparrow_{p_0}^q\Uparrow_{p_0}^{p}|\leqslant\frac\pi2$ for all $p\in F\cap U\setminus \{p_0\}$.
Then similar to the proof for (C3), we can see that $|qp_0|\leqslant \frac\pi2$ and the equality
implies that $|qp|=\frac\pi2$.
\end{proof}

As an example of the rigidity part in (C3) (and Proposition \ref{pro4.2}),
we have the following observation:

\begin{prop}\label{pro4.1}
Let $F$ be a locally quasi-convex subset in an $X\in \text{\rm Alex}(k)$, and let
$q\in X$. If $\dist_q|_F$ attains a local minimum
at $p\in F$ with $p\neq q$ and $\Sigma_p F\neq\emptyset$, then $|\eta\xi|=\frac{\pi}{2}$
for any $\eta\in \Uparrow_{p}^q$ and any $\xi\in \Sigma_p F$.
As a result,  for some $\bar p\in F$, if $\Sigma_{\bar p}F$ consists of a single point $\xi$,
then $|\zeta\xi|\leqslant\frac{\pi}{2}$ for any $\zeta\in\Sigma_{\bar p}X$.
\end{prop}

This proposition can be viewed as a generalized version of (\ref{eqn0.3}).
By (C1), $\Sigma_pF$ is quasi-convex in $\Sigma_pX$,
so the first statement of Proposition \ref{pro4.1} is a corollary
of (C3) if $\Sigma_pF$ contains at least two points. Anyway, we can provide
a proof directly from Definition A.

\begin{proof} For the first statement, note that there
a $[qp]$ such that $\eta=\uparrow_p^q$, and there is $p_i\in F$ such that
$p_i\to p$ and $\Uparrow_p^{p_i}\to\xi$ as $i\to\infty$. By the local quasi-convexity of $F$ (i.e. Definition A)
together with Proposition \ref{pro1.3}, we have that
$|\eta\Uparrow_{p}^{p_i}|\leqslant\frac\pi2$ for large $i$,
which implies that $|\eta\xi|\leqslant\frac\pi2$. On the other hand, by Lemma \ref{lem3.6}, we can conclude that $|\Uparrow_{p}^q\xi|\geqslant \frac{\pi}{2}$
because $p$ is a local minimal point of $\dist_q|_F$. It therefore follows that $|\eta\xi|=\frac{\pi}{2}$.

For the second statement, if there is $\zeta\in\Sigma_{\bar p}X$ such that $|\zeta\xi|>\frac{\pi}{2}$,
we can find a $[\bar q\bar p]$ with $\uparrow_{\bar p}^{\bar q}$ close to $\zeta$
satisfying $|\uparrow_{\bar p}^{\bar q}\xi|=|\Uparrow_{\bar p}^{\bar q}\xi|>\frac\pi2$. So, by Lemma \ref{lem3.6},
it holds that $|\bar p\bar q|<|p'\bar q|$ for $p'\in F$ sufficiently close to $\bar p$.
Namely, $\dist_{\bar q}|_F$ attains a local minimum at $\bar p$, and thus it follows from the first statement that
$|\uparrow_{\bar p}^{\bar q}\xi|=\frac{\pi}{2}$, a contradiction.
\end{proof}

\subsection{Applications of (C3)}

As an application of (C3) (and Proposition \ref{pro4.2}), we have the following observation on quasi-convex subsets in Alexandrov spaces with curvature $\geqslant1$.

\begin{lemma} \label{lem4.3}
Let $F\subset Y\in \text{\rm Alex}(1)$ be a quasi-convex subset containing at least two points,
and $q\in Y$ ($q$ may lie in $F$). If $\dist_q|_F$ attains a minimum at $p\in F$, then for any $p'\in F$,
$$|qp|+|qp'|\leqslant \pi;$$
and equality holds for some $p'$ if and only if  either $|pp'|=\pi$ or $|q\bar p|=\frac{\pi}{2}$ for all $\bar p\in F$.
\end{lemma}

Lemma \ref{lem4.3} might not be true if $F$ is just only locally quasi-convex.

\begin{proof} If $q\in F$, then the conclusion is trivial because the diameter of $Y$ is less than or equal to $\pi$ (\cite{BGP}).
It remains to consider the case where $q\notin F$. According to (C3), we know that $|qp|\leqslant \frac{\pi}{2}$.
Moreover, the quasi-convexity of $F$ implies that $\tilde\angle_1 qpp' \leqslant \frac{\pi}{2}$ for any $p'\in F\setminus\{p\}$ (see Proposition B). These enable us to view
the conclusion of the lemma as some basic facts on the unit sphere.
\end{proof}

Via Lemma \ref{lem4.3}, we can see a further application of (C3) when considering some kind of sum of $|q_1p|$ and $|q_2p|$ with $q_1,q_2\in Y$ and $p\in F$,
which will play a crucial role in the proof of (C2), i.e. the generalized Liberman theorem (see Lemma \ref{lem5.1} below).

\begin{lemma}\label{lem4.4}
Let $F\subset Y\in \text{\rm Alex}(1)$ be a quasi-convex subset containing at least two points, and
let $x_1, x_2\in Y$. Suppose that there is $a_1\geqslant 0$ and $a_2\geqslant 0$ such that
$a_1\cos |x_1\xi|+a_2\cos |x_2\xi|\leqslant 0$ for all $\xi\in F$. Then the following hold:

\vskip1mm

\noindent {\rm (4.4.1)} If $a_i>0$ and if there is $\xi_i\in F$ such that $|x_i\xi_i|=|x_iF|$
($i=1, 2)$ and $|\xi_1\xi_2|\neq\pi$, then we have that
$|x_1\xi|=|x_2\xi|=\frac\pi2$ for all $\xi\in F$.

\vskip1mm

\noindent {\rm (4.4.2)} We in fact have that $a_1\cos |x_1\xi|+a_2\cos |x_2\xi|=0$ for all $\xi\in F$.
	
\vskip1mm

As a corollary of {\rm (4.4.2)}, for any $\eta\in Y\setminus F$ and $\xi\in F$ with $|\eta\xi|=|\eta F|$, we have that

\vskip1mm

\noindent {\rm (4.4.3)}\ \	 $\text{\rm Sign}\left(a_1\cos |x_1\eta|+a_2\cos |x_2\eta|\right)=
	\text{\rm Sign}\left(a_1\sin|x_1\xi|\cos\tilde\angle_1 x_1\xi\eta+a_2\sin|x_2\xi|\cos\tilde\angle_1 x_2\xi\eta\right).$
\end{lemma}

Note that if $a_1=a_2>0$, then `$a_1\cos |x_1\xi|+a_2\cos |x_2\xi|\leqslant0$' is equivalent to
`$|x_1\xi|+|x_2\xi|\geqslant\pi$'. This together with (\ref{eqn5.8}) below
shows that the condition `$a_1\cos |x_1\xi|+a_2\cos |x_2\xi|\leqslant 0$' in Lemma \ref{lem4.4} is not
an artificial one. Moreover, note that if $\tilde \angle_1 x_i\xi\eta \leqslant \frac{\pi}{2}$ in (4.4.3),
then $a_1\cos |x_1\eta|+a_2\cos |x_2\eta|\geqslant 0$,
which is crucial to the proof of (C2) (see (\ref{eqn5.7}) and (\ref{eqn5.9}) below).

\begin{proof}[Proof of Lemma \ref{lem4.4}]\

\vskip1mm

\noindent (4.4.1)  Since $\xi_i\in F$ and $|x_i\xi_i|=|x_iF|$ for $i=1$ and 2, by Lemma \ref{lem4.3} we have that
$$
|x_1\xi_1|+|x_1\xi_2|\leqslant \pi \text{\quad  and \quad } |x_2\xi_1|+|x_2\xi_2|\leqslant \pi;
$$
or equivalently
\begin{equation}\label{eqn4.1}
 \cos|x_1\xi_1|+\cos|x_1\xi_2|\geqslant 0  \text{\quad  and \quad } \cos|x_2\xi_1|+\cos|x_2\xi_2|\geqslant 0.
\end{equation}
On the other hand, by the assumption we have that
$$
a_1(\cos |x_1\xi_1|+\cos |x_1\xi_2|)+ a_2(\cos |x_2\xi_1| + \cos |x_2\xi_2|) \leqslant 0  \text{ with } a_1,a_2>0,
$$
which together with (\ref{eqn4.1}) implies that
\begin{equation} \label{eqn4.2}
|x_1\xi_1|+|x_1\xi_2|=\pi  \text{\quad  and \quad }  |x_2\xi_1|+|x_2\xi_2|=\pi,
\end{equation}
and
\begin{equation} \label{eqn4.3}
a_1\cos |x_1\xi_1|+a_2\cos |x_2\xi_1|=0  \text{\quad  and \quad }  a_1\cos |x_1\xi_2|+a_2\cos |x_2\xi_2|=0.
\end{equation}
(So far, the proof has no relation with whether $|\xi_1\xi_2|=\pi$ or not.) Since $|\xi_1\xi_2|\neq\pi$,
(\ref{eqn4.2}) implies that $|x_1\xi|=|x_2\xi|=\frac\pi2$ for all $\xi\in F$ by Lemma \ref{lem4.3}.

\vskip1mm

\noindent (4.4.2)  It suffices to show that $a_1\cos |x_1\xi|+a_2\cos |x_2\xi|\geqslant 0$ for all $\xi \in F$.

Note that it is trivial if $a_1=a_2=0$. If only one of $a_i$, say $a_1$, is equal to 0,
then the assumption `$a_1\cos |x_1\xi|+a_2\cos |x_2\xi|\leqslant 0$' is equivalent to
$|x_2\xi|\geqslant\frac\pi2$ for all $\xi\in F$. By (C3), this  implies that
$|x_2\xi|=\frac\pi2$ for all $\xi\in F$, and thus $a_1\cos |x_1\xi|+a_2\cos |x_2\xi|=0$.

We now consider the remaining case where $a_1>0$ and $a_2>0$.
Let $\xi_1, \xi_2\in F$ such that $|x_i\xi_i|=|x_iF|$ for $i=1$ and $2$.
By (4.4.1), we can assume that $|\xi_1\xi_2|=\pi$. In this case, it is obvious that
$|\xi\xi_1|+|\xi\xi_2|=\pi$ for all $\xi \in F$, i.e.,
$$
\cos  |\xi\xi_1|+\cos |\xi\xi_2|=0.
$$
This plus the second equality of (\ref{eqn4.2}) and the first equality of (\ref{eqn4.3})  implies that
$$ a_1\cos |x_1\xi_1|\cos|\xi_1\xi|+a_2\cos |x_2\xi_2|\cos|\xi_2\xi|=0.$$
Meantime, by the quasi-convexity of $F$ (see Proposition B), we have that
\begin{equation} \label{eqn4.4}
\tilde\angle_1 x_i\xi_i\xi\leqslant\frac\pi2 \text{ for } i=1,2.
\end{equation}
Hence, by the Law of Cosine, it follows that for all $\xi\in F$
$$  a_1\cos |x_1\xi|+a_2\cos |x_2\xi|
=\sum_{i=1}^2 a_i(\cos |x_i\xi_i|\cos|\xi_i\xi|+\sin |x_i\xi_i|\sin|\xi_i\xi|\cos \tilde\angle_1 x_i\xi_i\xi)
 \geqslant 0.
$$

\noindent (4.4.3) Note that $|\eta\xi|>0$ because $\eta\not\in F$, and that $|\eta\xi|=|\eta F|<\pi$ by (C3).
And thus (4.4.3) follows from the following computation based on the Law of Cosine:
\begin{align*}
a_1\cos |x_1\eta|+a_2\cos |x_2\eta| & = \sum_{i=1}^{2} a_i(\cos|x_i\xi|\cos|\eta\xi|+\sin|x_i\xi|\sin|\eta\xi|\cos\tilde\angle_1 x_i\xi\eta)\\
& =\cos|\eta\xi|\sum_{i=1}^{2} a_i \cos|x_i\xi| +\sin|\eta\xi| \sum_{i=1}^{2} a_i \sin |x_i\xi|\cos\tilde\angle_1 x_i\xi\eta \\
& =\sin|\eta\xi| \sum_{i=1}^{2} a_i \sin |x_i\xi|\cos\tilde\angle_1 x_i\xi\eta,
\end{align*}
where the third equality is due to (4.4.2).
\end{proof}

\begin{remark}\label{rem4.6} {\rm
In Lemma \ref{lem4.4}, if $F$ consists of a single point $\xi$ and if $|\xi\eta|\leqslant\frac\pi2$ for all $\eta\in Y$,
then `$a_1\cos |x_1\xi|+a_2\cos |x_2\xi|\leqslant 0$' implies that $a_1\cos |x_1\xi|+a_2\cos |x_2\xi|=0$ automatically.
Hence, in this situation, the proof for (4.4.3) still goes through, i.e. (4.4.3) still holds.}
\end{remark}

\begin{remark}\label{rem4.7} {\rm
In Lemma \ref{lem4.4}, if $F$ is extremal in $Y$, note that (\ref{eqn4.4}) holds for all $\xi\in Y$ by (\ref{eqn0.2});
and thus we in fact have that $a_1\cos |x_1\xi|+a_2\cos |x_2\xi|\geqslant0$ for all $\xi\in Y$
from the proof of (4.4.2).
Indeed, this can be concluded from the arguments in \cite{PP1}. However, in our situation,
the key to (4.4.2) is (4.4.1) which relies on Lemma \ref{lem4.3}, while there is no version of (4.4.1) or Lemma \ref{lem4.3}
for extremal case in \cite{PP1} at all.}
\end{remark}

%%%%%%%%%%%%%%%%%%%%%%%%%%%%%%%%%%%%%% section 5 %%%%%%%%%%%%%%%%%%%%%%%%%%%%%%%%%%%%%%%%%%%%%%%%%%

\section{Proof for (C2) of Theorem C}

In this section, we mainly give a proof for (C2) of Theorem C, i.e. the generalized Liberman theorem,
in which we only consider the case where $k=0$ due to the similarity just as in \cite{PP1}.
To the proof for extremal case (\cite{PP1}), the key is the inequality in (\ref{eqn5.4}) below which is
the only geometry due to the extremal property in the proof.
And one can check that the proof in \cite{PP1} except the arguments for (\ref{eqn5.4})
still work for locally quasi-convex case.
Namely, in our situation, in order to prove (C2) we need only to verify that the inequality in (\ref{eqn5.4}) still holds.
Indeed, it is true by Lemma \ref{lem5.1} below, but the proof should be much more
technical as mentioned in Remark \ref{rem0.5}.

Before giving Lemma \ref{lem5.1}, for the convenience of readers, we had better provide
an outline of the proof for extremal case in \cite{PP1}.

First of all, we recall the definition of quasi-geodesics.
In an $X\in \text{\rm Alex}(k)$, let
$\gamma:[a,b]\to X$ be an arc-length parameterized curve, denoted by $\gamma(t)|_{[a,b]}$.
Recall that if $\gamma(t)|_{[a,b]}$ is a minimal geodesic,
`curvature $\geqslant k$' means that, for any $p\in X$ and $t\in (a,b)$ with $p\neq\gamma(t)$,
\begin{equation}\label{eqn5.1}
\tilde\angle_k p\gamma(t)\gamma(a)+\tilde\angle_k p\gamma(t)\gamma(b)\leqslant \pi\
(\text{i.e.\ $\cos\tilde\angle_k p\gamma(t)\gamma(a)+\cos\tilde\angle_k p\gamma(t)\gamma(b)\geqslant 0)$, \cite{BGP})}
\end{equation}
which, in the case where $k=0$, is equivalent to
\begin{equation}\label{eqn5.2}
(t-a)(|p\gamma(b)|^2-|p\gamma(t)|^2)+(b-t)(|p\gamma(t)|^2-|p\gamma(a)|^2)\leqslant (b-a)(t-a)(b-t).
\end{equation}

From now on, we only consider the case where $k=0$. In this case,
$\gamma(t)|_{[a,b]}$ is called to be a {\it quasi-geodesic} if (\ref{eqn5.2}) holds for all $p\in X$ and $t\in (a,b)$
(\cite{PP1}, refer to \cite{PP2} for an equivalent definition in terms of `development'). It is obvious that
quasi-geodesics are as good as geodesics in the sense of comparison geometry.
And it turns out that quasi-geodesics have to be geodesics in a Riemannian manifold (\cite{PP2}), so
quasi-geodesics are very worthy of research (\cite{Pet}).

We now begin to show the idea to prove that $\gamma(t)|_{[a,b]}$ is a quasi-geodesic if it is a shortest path in an extremal subset $F$ with induced intrinsic metric from $X$ (\cite{PP1}).
It turns out that the proof can be reduced to show that there exists
a shortest path $\bar\gamma(t)|_{[a,b]}\subset F$ jointing $\gamma(a)$ and $\gamma(b)$
such that it is a quasi-geodesic. In fact,
this implies that for any partition $a=t_0<t_1<t_2<\cdots<t_l=b$ of $[a,b]$
there is a quasi-geodesic (also a shortest path) in $F$ jointing $\gamma(t_i)$ and $\gamma(t_{i+1})$ for any $1\leqslant i\leqslant l-1$, based on which one can derive that $\gamma(t)|_{[a,b]}$ is a quasi-geodesic.
In order to show the existence of $\bar\gamma(t)|_{[a,b]}$,
for any $m\in\Bbb N^{+}$ we consider
\begin{equation}\label{eqn5.3}
S_m\triangleq\min_{a_0, a_1, \cdots, a_m\in F} m\sum_{i=0}^{m-1}|a_ia_{i+1}|^2,
\end{equation}
where $a_0=\gamma(a)$ and $a_m=\gamma(b)$. Since $\gamma(t)|_{[a,b]}$ is a shortest path jointing $\gamma(a)$ and $\gamma(b)$,
one can show that $S_m$ converges to $(b-a)^2$ as $m\to \infty$,
which implies that each $|a_ia_{i+1}|\to 0$ and $\frac{|a_ja_{j+1}|}{|a_ia_{i+1}|}\to 1$ for all $0\leqslant i,j\leqslant m-1$ as $m\to\infty$.
Let $S_m$ be realized by $\{a_0, a_1, \cdots, a_m\}\subset F$. Passing to a subsequence, we can assume that
$\{a_0, a_1, \cdots, a_m\}$ converges to a shortest path $\bar\gamma(t)|_{[a,b]}$ as $m\to\infty$ jointing $\gamma(a)$ and $\gamma(b)$ (where $t$ is also arc-length parameter). We claim that $\bar\gamma(t)|_{[a,b]}$ is just the desired path.

Note that the choice of $\{a_0, a_1, \cdots, a_m\}$ implies that, for each $1\leqslant i\leqslant m-1$,
the function $|ya_{i-1}|^2+|ya_{i+1}|^2$ with $y\in F$ attains a minimum at $a_i$.
Then by the extremal property of $F$, for any $[a_ip]$ with $p\in X$ and $p\neq a_i$, we can conclude that
\begin{equation}\label{eqn5.4}
|a_{i-1}a_i|\cos |\Uparrow_{a_i}^{a_{i-1}}\uparrow_{a_i}^{p}| +|a_ia_{i+1}|\cos |\Uparrow_{a_i}^{a_{i+1}}\uparrow_{a_i}^{p}|\geqslant 0,
\end{equation}
which implies that (by Theorem \ref{thm1.1}, note that we have assumed $k=0$)
\begin{equation}\label{eqn5.5}
\left(|pa_{i+1}|^2-|pa_{i}|^2\right)-\left(|pa_{i}|^2-|pa_{i-1}|^2\right)\leqslant |a_{i-1}a_i|^2+|a_{i}a_{i+1}|^2.
\end{equation}
(\ref{eqn5.4}) and (\ref{eqn5.5}) are equivalent to the corresponding (\ref{eqn5.1}) and (\ref{eqn5.2})
respectively if $|a_{i-1}a_i|=|a_{i}a_{i+1}|$ (note that $\frac{|a_ja_{j+1}|}{|a_ia_{i+1}|}\to 1$ as $m\to \infty$ for all $0\leqslant i,j\leqslant m-1$).
By putting the (\ref{eqn5.5}) for all $i$ together, and through
a really skillful calculation with a limiting argument, it can be concluded that
$\bar\gamma(t)|_{[a,b]}$ satisfies (\ref{eqn5.2}) for all $t\in (a,b)$ and $p\in X$ with $p\neq\gamma(t)$,
i.e. $\bar\gamma(t)|_{[a,b]}$ is a quasi-geodesic.

\vskip2mm

Let's move onto the case where $F$ is locally quasi-convex in $X$, and
show that a shortest path $\gamma(t)|_{[a,b]}$ in $F$ is a quasi-geodesic (i.e. (C2)).
In applying the above approach to our situation,
a main problem is that $\{a_0, a_1, \cdots, a_m\}$ might have no limit as $m\to\infty$ because $F$ might not be closed in $X$.
However, as mentioned above, we can divide $\gamma(t)|_{[a,b]}$
into sufficiently small pieces, and it suffices to prove that each piece is a quasi-geodesic. Since each piece
can be sufficiently small, one can assume that it lies in a `quasi-convex neighborhood' $U$ of some $\gamma(t)$ (see Definition A for the choice of $U$). Namely, (C2) can be reduced to the case where $F$ is a quasi-convex subset.
In such a situation, $\{a_0, a_1, \cdots, a_m\}$ converges to a shortest path in $F$ as $m\to\infty$.
Moreover, as mentioned in the beginning of this section,
we need only to show that (\ref{eqn5.4}) is still true in the premise of
`$|ya_{i-1}|^2+|ya_{i+1}|^2$ with $y\in F$ attains a minimum at $a_i$'.
This is guaranteed by Lemma \ref{lem5.1} below (see (5.1.2) and Remark \ref{rem5.2}).

\begin{lemma}\label{lem5.1}
Let $F$ be a quasi-convex subset in an $X\in \text{\rm Alex}(k)$, and let $q_1,q_2\in X$.
Suppose that $f(|yq_1|)+f(|yq_2|)$ with $y\in F$ attains a local minimum at $p\in F$,
where the function $f(t)$ satisfies $f'(t)|_{t=|yq_i|}\geqslant 0$ ($i=1,2$).
Then the following hold:

\vskip1mm

\noindent{\rm (5.1.1)} If $\Sigma_pF\neq\emptyset$, then for any $[pq_1]$ and $[pq_2]$ and for all $\xi \in \Sigma_pF$ we have that
\begin{equation}\label{eqn5.6}
f'(|pq_1|)\cos |\uparrow_{p}^{q_1}\xi| +f'(|pq_2|)\cos |\uparrow_{p}^{q_2}\xi|=0,
\end{equation}
which implies that if $f'(|pq_i|)>0$, then
\begin{equation}\label{eqn5.6.0}
\text{$|\uparrow_{p}^{q_i}\xi|=|\Uparrow_{p}^{q_i}\xi|$}.
\end{equation}

\vskip1mm

\noindent{\rm (5.1.2)} If both $q_1$ and $q_2$ lie in $F$, then for any $[px]$ with $x\in X$ and $p\neq x$ we have that
	\begin{equation}\label{eqn5.7}
	f'(|pq_1|)\cos |\Uparrow_{p}^{q_1}\uparrow_{p}^{x}| +f'(|pq_2|)\cos |\Uparrow_{p}^{q_2}\uparrow_{p}^{x}|\geqslant 0.
	\end{equation}
\end{lemma}

\begin{remark}\label{rem5.2} { \rm  When applying (5.1.2) to see (\ref{eqn5.4}), we just need
to set $f(t)=t^2$ in Lemma \ref{lem5.1} (in fact, the case where $f'(|pq_i|)=0$
is not involved).  As mentioned above, we just only give a proof
of (C2) for the case where $k=0$. For other cases, we should set $f(t)=1-\cos(\sqrt{k}t)$
or $\cosh(\sqrt{-k}t)-1$ for $k>0$ or $k<0$ respectively in order to
see the corresponding (\ref{eqn5.4}).}
\end{remark}

\begin{proof}[Proof of Lemma \ref{lem5.1}] \

(5.1.1)\ \ Since $p$ is a local minimum point of $f(y)|_{y\in F}$ and
$\Sigma_pF\neq\emptyset$, by Lemma \ref{lem3.6} we have that
$$
f'(|pq_1|)\cos |\Uparrow_{p}^{q_1}\xi| +f'(|pq_2|)\cos |\Uparrow_{p}^{q_2}\xi|\leqslant 0 \text{ for all } \xi \in \Sigma_pF;
$$
and thus, for any $[pq_1]$ and $[pq_2]$,
\begin{equation}\label{eqn5.8}
f'(|pq_1|)\cos |\uparrow_{p}^{q_1}\xi| + f'(|pq_2|)\cos |\uparrow_{p}^{q_2}\xi|\leqslant 0.
\end{equation}
Note that $\Sigma_pF$ is quasi-convex in $\Sigma_pX$ (by (C1)),
and that $|\eta\xi|\leqslant\frac\pi2$ for all $\eta\in\Sigma_pX$ if $\Sigma_pF=\{\xi\}$ (by Proposition \ref{pro4.1}).
Then (\ref{eqn5.8}) enables us to apply (4.4.2) and Remark \ref{rem4.6} to conclude (\ref{eqn5.6}).
Moreover, due to the arbitrariness of $[pq_i]$,  it follows that
$|\uparrow_{p}^{q_i}\xi|=|\Uparrow_{p}^{q_i}\xi|$ if $f'(|pq_i|)>0$.

\vskip1mm

(5.1.2)\ \ We first consider the case where $\Sigma_pF\neq\emptyset$.
For any $[px]$ with $p\neq x$, if $\uparrow_p^x\in\Sigma_pF$, then (\ref{eqn5.7}) follows from (\ref{eqn5.6}).
If $\uparrow_p^x\notin\Sigma_pF$, then there is $\xi\in \Sigma_pF$ such that $|\uparrow_p^x\xi|=|\uparrow_p^x \Sigma_pF|$.
By (4.4.3) and Remark \ref{rem4.6},
to see (\ref{eqn5.7}) it suffices to show that if $f'(|pq_i|)>0$,
then there is $[pq_i]$ such that
\begin{equation}\label{eqn5.9}
\tilde\angle_1 \uparrow_{p}^{q_i}\xi\uparrow_{p}^x\leqslant\frac\pi2.
\end{equation}
Here, a key is that $|\uparrow_{p}^{q_i}\xi|=|\Uparrow_{p}^{q_i}\xi|$ for all
$\uparrow_{p}^{q_i}\in \Uparrow_{p}^{q_i}$ if $f'(|pq_i|)>0$  (by (\ref{eqn5.6.0})).
This together with $q_i\in F$ enables us to apply Corollary \ref{cor3.5} to see (\ref{eqn5.9}).

As for the case where $\Sigma_pF=\emptyset$, i.e. $p$ is an isolated point of $F$, it is almost obvious that $|\Uparrow_{p}^{q_i}\uparrow_p^{x}|\leqslant\frac\pi2$ which implies (\ref{eqn5.7}).
In fact, $\dist_{x}|_F$ attains a local minimum at $p$ in this case, so by Proposition \ref{pro1.3} we have that $|\Uparrow_{p}^{q_i}\uparrow_p^{x}|\leqslant\frac\pi2$.
\end{proof}

\begin{remark}\label{rem5.3} { \rm
According to Remark \ref{rem4.7}, if $F$ is extremal, then (\ref{eqn5.8}) enables us to
conclude that $f'(|pq_1|)\cos |\uparrow_{p}^{q_1}\xi|+f'(|pq_2|)\cos |\uparrow_{p}^{q_2}\xi|\geqslant0$
for all $\xi \in \Sigma_pX$, which relies on the proof of (4.4.2).
Consequently, restricted to extremal case, (\ref{eqn5.7}) (i.e. (\ref{eqn5.4})) holds without involving (4.4.3) and Corollary \ref{cor3.5},  and our proof for it is different from that in \cite{PP1} (see Remark \ref{rem4.7}). }
\end{remark}

%%%%%%%%%%%%%%%%%%%%%%%%%%%%%%%%%%%%%%%%%%%%%%%%%%%%%%%%%%%%%%%

\noindent School of Mathematical Sciences (and Lab. math. Com.
Sys.), Beijing Normal University, Beijing, 100875
P.R.C.\\
e-mail: suxiaole$@$bnu.edu.cn; wyusheng$@$bnu.edu.cn

\vskip2mm

\noindent Mathematics Department, Capital Normal University,
Beijing, 100037 P.R.C.\\
e-mail: 5598@cnu.edu.cn

\end{document}